\documentclass[hidelinks,11pt]{amsart}
\usepackage{amsmath,amsthm,amscd,amssymb}
\usepackage{geometry}
\usepackage{subcaption}
\usepackage{pb-diagram}
\usepackage{graphicx}
\usepackage{setspace}
\usepackage{enumerate}

\usepackage{pstricks}
\usepackage{pst-node}
\usepackage{url}

\usepackage{xcolor}
\usepackage{hyperref}

\usepackage{mathrsfs}
\usepackage{MnSymbol}
\usepackage{gensymb} 
\usepackage{enumerate}
\usepackage{mathtools} 
\DeclareMathOperator{\arf}{Arf}

\newtheorem{thm}{Theorem}[section]

\newtheorem{lem}[thm]{Lemma}
\newtheorem{prop}[thm]{Proposition}
\newtheorem{cor}[thm]{Corollary}

\theoremstyle{definition}

\newtheorem{rem}[thm]{Remark}

\numberwithin{equation}{section}
\numberwithin{figure}{section}

\newcommand{\Z}{\mathbb{Z}}
\newcommand{\Q}{\mathbb{Q}}
\newcommand{\DK}{D_{K}(S^{3})}

\usepackage{titlesec}
\titleformat{\section} {\normalfont\scshape\bfseries\filcenter}{\thesection }{1em}{}
\titleformat{\subsection} {\normalfont\scshape\bfseries\filcenter}{\thesubsection}{1em}{}

\begin{document}

\title[{NON-ORIENTABLE 4-GENUS OF TORUS KNOTS}]{NON-ORIENTABLE 4-GENUS OF TORUS KNOTS}

\author[MEGAN FAIRCHILD, HAILEY JAY GARCIA, JAKE MURPHY, HANNAH PERCLE]{MEGAN FAIRCHILD, HAILEY JAY GARCIA, JAKE MURPHY, HANNAH PERCLE}
\address{Department of Mathematics \\ Louisiana State University}
\email{mfarr17@lsu.edu}
\email{jgarc86@lsu.edu}
\email{jmurp61@lsu.edu}
\email{hpercl1@lsu.edu}

\maketitle

\textbf{ABSTRACT.} The non-orientable 4-genus of a knot $K$ in $S^{3}$, denoted $\gamma_4(K)$, measures the minimum genus of a non-orientable surface in $B^{4}$ bounded by $K$. We compute bounds for the non-orientable 4-genus of knots $T_{5, q}$ and $T_{6, q}$, extending previous research. Additionally, we provide a generalized, non-recursive formula for $d(S^{3}_{-1}(T_{p,q}))$, the $d$-invariant of -1-surgery on torus knots.

\section{INTRODUCTION}

Given a knot $K$ in $S^3$, it is a well-studied problem to determine the minimum genus of an orientable surface in $B^{4}$ with boundary $K$. This is called the orientable 4-genus of $K$ and was first studied by Fox and Milnor when they defined the slice knot in 1958, that is, knots in $S^{3}$ that bound a smoothly embedded disk in $B^{4}$. We cite their 1966 paper \cite{fox1966singularities}, as the 1958 reference is from unpublished work. Since then, there has been steady progress in determining the orientable 4-genus of knots in $S^3$. In 2000, Murakami and Yasuhara \cite{MY} introduced the non-orientable four genus, denoted $\gamma_4(K)$. This is defined as the minimum genus of a non-orientable surface $F$ properly embedded in $B^{4}$ so that the boundary of $F$ is the knot $K$.

The non-orientable 4-genus of knots with 8 or 9 crossings was computed by Jabuka and Kelly \cite{JK}. Ghanbarian in \cite{N10} expanded this to knots with 10 crossings, and the first author determined the non-orientable 4-genus of non-alternating 11 crossing knots \cite{fairchild2024non}. Additionally, the non-orientable 4-genus of many double twist knots has been computed by Hoste, Shanahan, and Van Cott in \cite{dbltwist}. The Milnor conjecture, proved by Kronheimer and Mrowka in 1993 \cite{milnorconj}, states that the smooth 4-genus of a torus knot $T_{p, q}$ is $\frac{1}{2}(p-1)(q-1)$. Andrew Lobb proved that the non-orientable analog of the Milnor conjecture is false by showing that the non-orientable 4-genus of $T_{4,9}$ is 1 in~\cite{lobb2019counterexample}. Jabuka and Van Cott found additional counterexamples in \cite{NOmilnor}.

In this paper, we focus on the non-orientable 4-genus of torus knots. These knots are denoted $T_{p, q}$, where $p$ and $q$ are relatively prime. The non-orientable 4-genus of torus knots has been computed for all knots $T_{2, q}$ and $T_{3, q}$ by Allen \cite{Allen}, and most knots $T_{4, q}$ by Binns, Kang, Simone, Tru\"ol, and Sabloff \cite{torus1, torus2}. We aim to extend this research by computing the non-orientable 4-genus for torus knots $T_{5, q}$ and $T_{6, q}$. We explore various techniques in calculating $\gamma_{4}(T_{p, q})$ via computation of other knot invariants. One such invariant is the $d$-invariant from Heegaard Floer homology, which is computable by a recursive formula using the Alexander polynomial given in \cite{batson, OSref}. We update this to provide a general formula the $d$-invariant of torus knots.

\begin{thm}
\hyperref[dinvTHM]{For a torus knot $T_{p, q}$ where $p<q$,
    \[ d(S^{3}_{-1}(T_{p,q}))=2\left(\left\lfloor\dfrac{p}{2} \right\rfloor+\sum_{k=0}^{\lfloor\frac{p}{2} \rfloor-1}\left\lfloor\frac{(p-1-2k)q-p-1}{2p}\right\rfloor\right). \]}
\end{thm}

Knot invariants such as the $d$-invariant, signature, and Arf invariant provide lower bounds for the non-orientable 4-genus \cite{torus1}, and we computed these to obtain the following results. 

\begin{thm}
    For torus knots $T_{6, q}$, when $q \neq 1$, 
    \begin{enumerate}[(i)]
        \item if $q \equiv 5, 7, 11 \pmod{12}$, then $\gamma_{4}(T_{6, q}) \in \{ 2, 3 \}, $
        \item if $q \equiv 1 \pmod{12}$, then $\gamma_{4}(T_{6, q}) \in \{2, 3 \}$ if $q\equiv0\pmod{5}$, otherwise $\gamma_{4}(T_{6, q}) \in \{ 1, 2, 3\}$.
    \end{enumerate}
\end{thm}

Other techniques to finding the non-orientable 4 genus include band surgery and examining the linking form on the double branch cover, which are explored to give the following result. 

\begin{thm}
\hyperref[5qthm]{For torus knots $T_{5, q}$,} 
    \begin{enumerate}[(i)]
        \item \hyperref[5qthm]{If $q \equiv \pm 2 \pmod{5}$, then $\gamma_{4}(T_{5, q}) = 1,$}
        \item \hyperref[5qthm]{If $q \equiv 4,6,9 \pmod{10}$, then $\gamma_{4}(T_{5, q}) = 2.$}
        \item \hyperref[5qthm]{If $q \equiv 1 \pmod{10}$, then $\gamma_{4}(T_{5, q}) \in \{1,2\}.$}
    \end{enumerate}
\end{thm}

\textbf{Acknowledgements.} This material is based upon work supported by the National Science Foundation under Award No.~2231492 and Award No.~1907654. We thank Shea Vela-Vick and Angela Wu for helpful conversations.

\section{BACKGROUND}

To begin, we review knot invariants and examine the lower bounds on the non-orientable 4 genus coming from knot invariants. It should be noted that the smallest non-orientable 4 genus for a knot $K$ in $S^{3}$ is $\gamma_{4} (K) = 1$ meaning $K$ bounds a M\"obius band. We focus on torus knots $T_{p,q}$ in our paper, which are knots which smoothly embed on the torus, twisting $p$ times around the meridian and $q$ times around the longitude. Note that $T_{p, q} = T_{q, p}$.

\subsection{Knot Invariants}

The signature of a knot $\sigma (K)$ is defined to be the signature of the knot's Seifert matrix added to its transpose, $\sigma (V + V^{t})$ \cite{knotinfo}. Gordan, Litherland, and Murasugi provided a recursive formula for calculating the signature of torus knots, which we restate here:

\begin{thm}[Theorem 5.2 in \cite{sigma}]\label{sig}
		Let $p,q>0$. Note $\sigma(p, q) = \sigma(T_{-p, q}) = - \sigma(T_{p, q})$.
		\begin{enumerate}[(i)]
			\item If $2q<p$:
			\begin{itemize}
				\item If $q$ is odd, $\sigma(p,q)=\sigma(p-2q,q)+q^2-1$,
				\item if $q$ is even, $\sigma(p,q)=\sigma(p-2q,q)+q^2$.
			\end{itemize}
			\item If $q\le p \le 2q$:
			\begin{itemize}
				\item If $q$ is odd, $\sigma(p,q)=-\sigma(2q-p,q)+q^2-1$,
				\item if $q$ is even, $\sigma(p,q)=-\sigma(2q-p,q)+q^2-2$.
			\end{itemize}
				\item $\sigma(2q,q)=q^2$, $\sigma(p,1)=0$, $\sigma(p,2)=p-1$, $\sigma(p,q)=\sigma (q,p)$.
		\end{enumerate}
	\end{thm}

 The Arf invariant of a knot is denoted $\arf(K)$ and is a concordance invariant in $\mathbb{Z}_{2}$ and calculated using the Seifert form of a knot \cite{knotinfo}. For torus knots, the Arf invariant is straightforward. 
 
 \begin{rem}[Remark 2.2 in \cite{torus1}]\label{ArfRem}
		For odd $q$, $\arf(T(p,q))=\begin{cases}0& p\text{ odd or } q\equiv \pm 1 \pmod{8},\\ 
  1&p\text{ even and } q\equiv \pm 3 \pmod{8}.\end{cases}$
\end{rem}

The following proposition from Yasuhara \cite{yasuhara} was proved and used in the work of Gilmer and Livingston \cite{GL}, and seen in application by Jabuka and Kelly \cite{JK}, Ghanbarian \cite{N10}, and the first author \cite{fairchild2024non}. 

\begin{prop}[Proposition 5.1 in \cite{yasuhara}]\label{sigArf}
    Given a knot $K$ in $S^{3}$, if $\sigma(K) + 4  \arf(K) \equiv 4 \pmod{8} $, then $\gamma_{4}(K) \geq 2$.
\end{prop}

The Upsilon invariant of a knot $K$, $\Upsilon_{K}(t)$ for $t \in [0, 2]$, is a powerful tool from knot Floer homology. We will not discuss knot Floer homology, but the interested reader is referred to the overview by Ozv\'ath and Szab\'o \cite{hfkBackground}. We denote $\upsilon(K) = \Upsilon_{K}(1)$ and note this gives the Upsilon invariant in the non-orientable case. Specifically, unoriented knot Floer homology is defined and explored in detail by Ozsv{\'a}th, Stipsicz, and Szab{\'o} in \cite{OS2}. For torus knots, we will denote $\upsilon(T_{p, q})$ as $\upsilon(p, q)$. Ozsv{\'a}th, Stipsicz, and Szab{\'o} give us a recursive formula to calculate the generic upsilon function for torus knots in~\cite{OS1}, which we restate here for the specific case of $t=1$.

\begin{prop}[Proposition 6.3 in~\cite{OS1}]
For $T_{p,q}$ and $i\in\Z$ such that $1 \in \left[ \dfrac{2i}{p}, \dfrac{2i + 1}{p} \right]$,
\begin{equation}
    \upsilon(p, p+1) = -i(i+1) - \frac{1}{2}p(p-1-2i).
\end{equation}
\end{prop}

We also have the following recursive equation from Feller and Krcatovich,
\begin{prop}[Proposition 2.2 in~\cite{FK}] For $p<q$,
\begin{equation}\label{upsilonrecursion}
    \upsilon(p, q) = \upsilon(p, q-p) + \upsilon(p, p+1).
\end{equation} 
\end{prop}

From the signature and upsilon invariant, we get the following lower bound on $\gamma_4$.

\begin{thm}[Theorem 1.2 in \cite{OS2}]\label{UpsSigBound}
  For a knot $K$ in $S^{3}$, 
    $$\left| \upsilon(K) + \frac{\sigma(K)}{2} \right| \leq \gamma_{4} (K)$$

\end{thm}

Another invariant which comes from knot Floer homology is the $d$-invariant, which is denoted $d(S^{3}_{-1}(K))$ for -1-surgery on a knot $K$. An interested reader is referred to \cite{OSref} in which Ozv\'ath and Szab\'o discuss the construction of this invariant and provide a method to compute it for torus knots using the Alexander polynomial. In \cite{batson}, Batson discusses how to use the $d$-invariant to create a lower bound for the non-orientable 4-genus. 

The Alexander polynomial of a knot is a Laurent polynomial and is denoted $\Delta_{K}(t)$. The formal definition using crossing relations and resolutions (Skein relations) can be found in many introduction to knot theory texts, but we refer our reader to \cite{alexpoly}. One important thing to note for this paper is the Alexander polynomial of a knot is symmetric in the sense that $\Delta_K(t)=\Delta_K(t^{-1})$.

 For ease of notation, we denote the Alexander polynomial of $T_{p,q}$ by $\Delta_{p,q}(t)$. The Alexander polynomial of a torus knot can be computed from the following equation \cite{batson}. 
\begin{equation}\label{symmAlexPoly}
    \Delta_{p, q} (t) = t^{-(p-1)(q-1)/2} \dfrac{(1-t)(1-t^{pq})}{(1 - t^{p})(1-t^{q})}
\end{equation}

Note that the $t^{-(p-1)(q-1)/2}$ term in equation \ref{symmAlexPoly} is a symmetrizing term. That is, it has no effect on the coefficients of the polynomial, but shifts the powers of $t$ so that $\Delta_{p,q}$ is symmetric. For ease of calculations, we will frequently consider $t^{(p-1)(q-1)/2}\Delta_{p, q}(t)$ instead of $\Delta_{p,q}(t)$.

We show in section 3 how to use the Alexander polynomial of torus knots to calculate the $d$-invariant of a knot $K$, denoted $d(S^3_{-1}(K))$, in order to use the following lower bound on $\gamma_4(T_{p,q})$ from Batson.
 \begin{thm}[Theorem 1 in~\cite{batson}]\label{boundsd}
 For a knot $K$,
 \begin{equation*}
		\dfrac{\sigma(K)}{2}-d(S^{3}_{-1}(K)) \le \gamma_4(K).
	\end{equation*}
\end{thm}

\subsection{Non-Oriented Band Moves}

To perform a band move on an oriented knot $K$, we first take an embedding $\varphi$ of an oriented band $B=[0,1]\times[0,1]$ into $S^3$ such that $B\cap K=\varphi([0,1]\times\{0,1\})$. We then create a new link $L=(K\cup\varphi({\{0,1\}\times[0,1]))\setminus\varphi([0,1]\times\{0,1\})}$. If the orientation of either $[0,1]\times\{0\}$ or $[0,1]\times\{1\}$ disagrees with $K$, then this is an non-oriented band move. Otherwise, it is an oriented band move. Figure~\ref{bandmovecompare} shows an example of an orientable and a non-oriented band move. Note that an orientable band move applied to a knot results in a 2-component link, while an non-oriented band move results in a knot. This procedure is sometimes referred to as a pinch move, such as in \cite{torus1}. The minimum number of non-oriented band moves required to transform $T_{p, q}$ into the unknot is denoted $ \vartheta(T_{p, q}).$

\begin{figure}[ht]
    \centering
    \includegraphics[width = 5in]{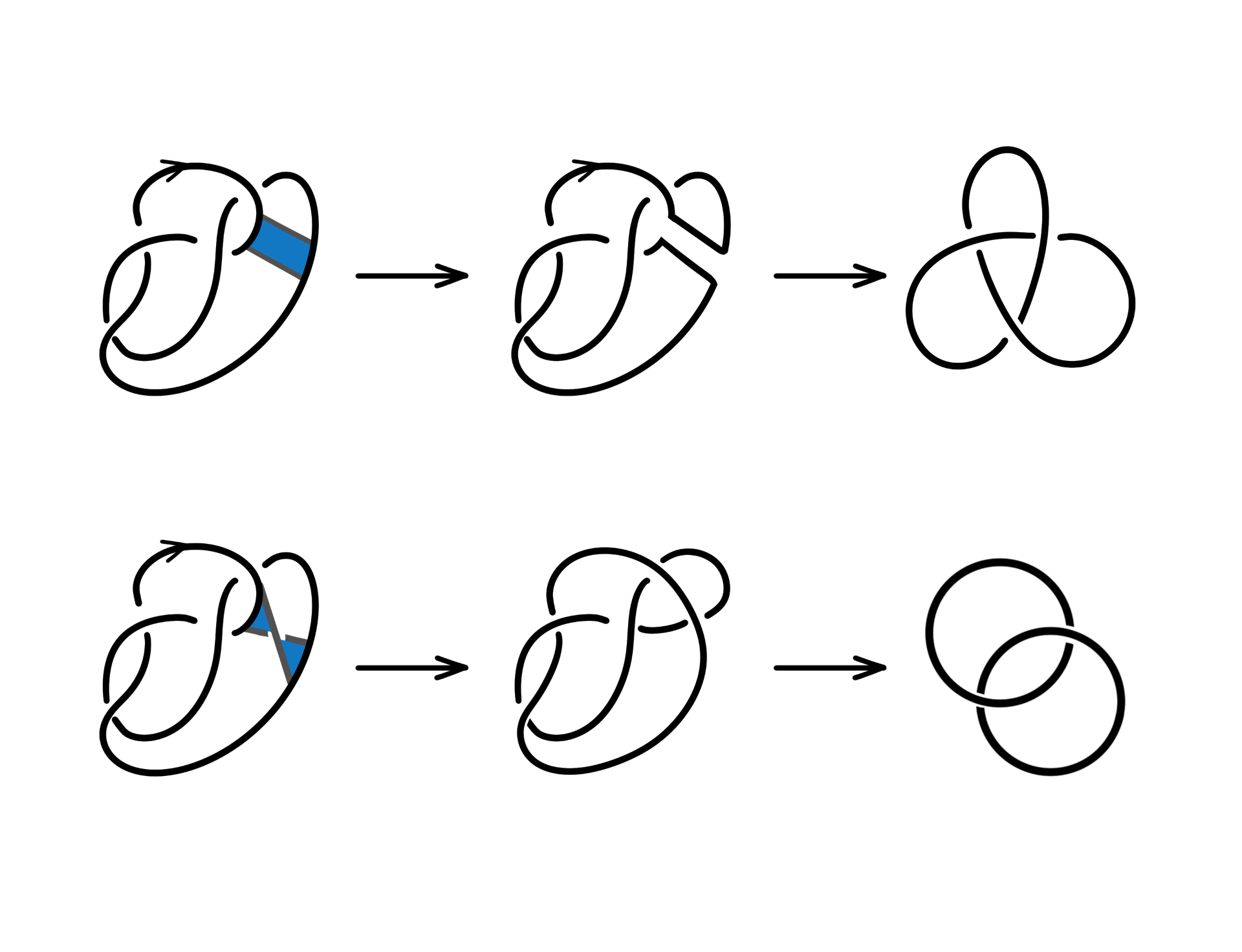}
    \caption{Top: Non-Oriented Band Move. Bottom: Oriented Band Move.}
    \label{bandmovecompare}
\end{figure}

\begin{thm}[Theorem 1.6 in \cite{torus1}]\label{pinchThm}

    Fix $p>3$. If $q \equiv p-1, $ $ p+1 $, or $2p+1 \pmod{2p} $, then $\vartheta(T_{p, q}) -1 \leq \gamma_{4} (T_{p, q}) \leq \vartheta (T_{p, q}) $. In particular, we have the following.
    \begin{enumerate}[(i)]
        \item If $p$ is odd: \\
        If $q \equiv p-1 \pmod{2p} $, then $\gamma_{4} (T_{p, q}) \in \{ \frac{p-3}{2}, \frac{p-1}{2} \}$, \\
        If $q \equiv p+1$ or $2p-1 \pmod{2p}$, then $\gamma_{4} (T_{p, q}) = \frac{p-1}{2}.$

        \item If $p$ is even: \\
        If $q > p$ and $q \equiv p-1, p+1 $ or $2p-1 \pmod{2p} $, then $ \gamma_{4}(T_{p, q}) \in \{ \frac{p-2}{2}, \frac{p}{2} \}.$
    \end{enumerate}
\end{thm}

\begin{rem}\label{6and11mod10}
    By Theorem \ref{pinchThm}, we have 
\begin{enumerate}[(i)]
    \item $\gamma_{4}(T_{5, q}) = 2$ when $q \equiv 6 $ or $9 \pmod{10},$
    \item $ \gamma_{4} (T_{6, q}) \in \{2, 3\} $ for $q \equiv 5, 7, $ or $11 \pmod{12}.$
\end{enumerate}

\end{rem}

We also have the following result from Jabuka and Kelly.
\begin{prop}[Proposition 2.4 in \cite{JK}]\label{BMbound}

If the knots $K$ and $K'$ are related by a non-oriented band move, then 
$$ \gamma_{4}(K) \leq \gamma_{4}(K') + 1 $$

If a knot $K$ is related to a slice knot $K'$ by a non-oriented band move, then $\gamma_{4}(K) = 1$.
    
\end{prop}

We now discuss the results that are most useful in determining $\gamma_{4}(T_{5, q})$. 

\begin{lem}[Lemma 2.9 in \cite{torus1}]\label{pinchBounds}
    Let $p \geq 2$ and $k \geq 1$. Performing a non-oriented band move on the torus knot $T_{p, kp \pm 1}$ yields the torus knot $T_{p-2, k(p-2)\pm 1}$. Consequently,  
    $$\vartheta(T_{p, kp \pm 1})=\begin{cases}
			 \frac{p-1}{2} &\text{ if $p$ is odd,}  \\
		   \frac{p}{2} &\text{ if $p$ is even and } kp \pm 1 \neq p-1,  \\
        \frac{p-2}{2} &\text{ if $p$ is even and } kp \pm 1 = p-1.
		\end{cases}$$
\end{lem}

This immediately gives us $\gamma_4(T_{6,q})\leq 3$ for all $q$.

\begin{prop}[Proposition 1.6 in \cite{NOmilnor}]
    The torus knot $T_{pk \pm 2, p}$ bounds a M\"obius band if $p \geq 3$ is odd and $k \geq 0$.
\end{prop}

\begin{rem}\label{pm2mod5}
    We thus conclude that the torus knots $T_{5, 5 k \pm 2}$ have $\gamma_{4} (T_{5, 5 k \pm 2}) = 1$ for all $k$, which proves part (i) of Theorem \ref{5qthm}. From Lemma \ref{pinchBounds} we also have $\gamma_{4}(T_{5, 5k \pm 1} ) \leq 2$ for any $k \geq 1$. We note that this gives us all results for $\gamma_{4}(T_{5, q})$ except for $q \equiv 1, 4 \pmod{10}$. 
\end{rem}


\section{RESULTS}

We first compute the signature of $T_{5, q}$ and $T_{6,q}$ torus knots so that we can apply the bound from Proposition~\ref{sigArf}. Recall the convention $$\sigma(p,q)=\sigma (T_{-p,q})=-\sigma (T_{p,q}).$$
	
	\begin{prop}\label{sigcalc}
		$$\sigma(6,q)=\begin{cases}
			18k& \text{ if } q=6k+1\ge 0,\\
			18k+16& \text{ if } q=6k+5\ge 0,
		\end{cases}$$
  and
		$$\sigma(5,q)=\begin{cases}
			24k& \text{ if } q=10k+1\ge 0,\\
			24k+8& \text{ if } q=10k+4\ge 0\\
			24k+16& \text{ if } q=10k+6\ge 0,\\
			24k+24& \text{ if } q=10k+9\ge 0,\\
			12k+4& \text{ if } q=5k+2\ge 0,\\
			12k+8& \text{ if } q=5k+3\ge 0.
		\end{cases}$$
	\end{prop}
	\begin{proof}
		We will prove this for $\sigma(6,q)$ by induction on $k$. See that $\sigma ({6,6(0)+1})=0$, $\sigma ({6,6(0)+5})=16$, $\sigma ({6,6(1)+1})=18$, $\sigma ({6,6(1)+5})=34$. Assume that the formula holds for some $k\ge 2$. We then have two cases: $q=6(k+1)+1$ and $q=6(k+1)+5$. The first case:
		\begin{align*}
			\sigma(6,6k+1) & =\sigma(6k+1,6)       \\
			& =\sigma(6k+1-12,6)+36 \\
			& =\sigma(6(k-2)+1,6)+36    \\
			& =18(k-2)+36               \\
			& =18k,
		\end{align*}
		and similarly, the second:
		\begin{align*}
			\sigma(6,6k+5) & =\sigma(6k+5,6)       \\
			& =\sigma(6k+5-12,6)+36 \\
			& =\sigma(6(k-2)+5,6)+36    \\
			& =18(k-2)+16+36              \\
			& =18k+16.
		\end{align*}
		The case for $\sigma(5,q)$ follows similarly.
	\end{proof}

	Using Remark~\ref{ArfRem} and our computation from Proposition~\ref{sigcalc}, we will compute the value of $\sigma(p,q)+4\arf(p,q)\pmod{8}$ for $p=5,6$.
	\vspace{1mm}

\begin{table}[ht]
\centering
	\begin{tabular}{c|c}
		$q\pmod{10}$ & $\sigma(5,q)+4\arf(5,q)\pmod{8}$ \\ \hline
		$1$     &            $0$            \\
		$2$     &            $0$            \\
		$3$     &            $0$            \\
		$4$    &            $4$            \\
		$6$    &            $4$            \\
		$7$    &            $0$            \\
		$8$    &            $0$            \\
		$9$    &            $0$
	\end{tabular}
     \caption{Remainders of $q\pmod{10}$ and values of $\sigma(5,q)+4\arf(5,q)\pmod{8}$.}
    \label{fig:sigarftable5}
\end{table}

\begin{table}[ht]
\centering
	\begin{tabular}{c|c}
		$q\pmod{12}$ & $\sigma(6,q)+4\arf(6,q)\pmod{8}$ \\ \hline
		$1$     &            $0$            \\
		$5$     &            $4$            \\
		$7$     &            $2$            \\
		$11$    &            $6$            
	\end{tabular}
     \caption{Remainders of $q\pmod{12}$ and values of $\sigma(6,q)+4\arf(6,q)\pmod{8}$.}
    \label{fig:sigarftable6}
\end{table}
\vspace{0.1in}

 \begin{cor}\label{SigArf5mod12}
     $\gamma_4(T_{6,12k+5}) \geq 2$ and $\gamma_4(T_{5,10k+4})=2$ for any integer $k$.
 \end{cor}

\begin{proof}
By Proposition~\ref{sigArf} and our computations in Tables~\ref{fig:sigarftable5} and~\ref{fig:sigarftable6}, we can see that $\gamma_4(T_{6,12k+5}) \geq 2$ and $\gamma_4(T_{5,10k+4})\geq2$. By Theorem~\ref{pinchThm} we have that $\gamma_4(T_{5,10k+4})=2$.
\end{proof}

With this, we are able to prove Theorem 1.3, which we restate below.

    {\noindent\bfseries Theorem 1.3.}\label{5qthm}
    \textit{For torus knots $T_{5, q}$, }
    \begin{enumerate}[(i)]
        \item \hyperref[5qthm]{If $q \equiv \pm 2 \pmod{5}$, then $\gamma_{4}(T_{5, q}) = 1,$}
        \item \hyperref[5qthm]{If $q \equiv 4,6,9 \pmod{10}$, then $\gamma_{4}(T_{5, q}) = 2.$}
        \item \hyperref[5qthm]{If $q \equiv 1 \pmod{10}$, then $\gamma_{4}(T_{5, q}) \in \{1,2\}.$}
    \end{enumerate}

\begin{proof}
    Note that part(i) follows from Remark~\ref{pm2mod5}, part (ii) follows from Remark~\ref{6and11mod10} and Corollary~\ref{SigArf5mod12}, and part (iii) follows from Lemma~\ref{pinchBounds}.
\end{proof}

\begin{lem}\label{BMknots}
    The knots $T_{6, 5}$ and $T_{6, 17}$ both have non-orientable 4-genus $2$.
\end{lem}

\begin{figure}[ht]
    \centering
    \includegraphics[width = 5in]{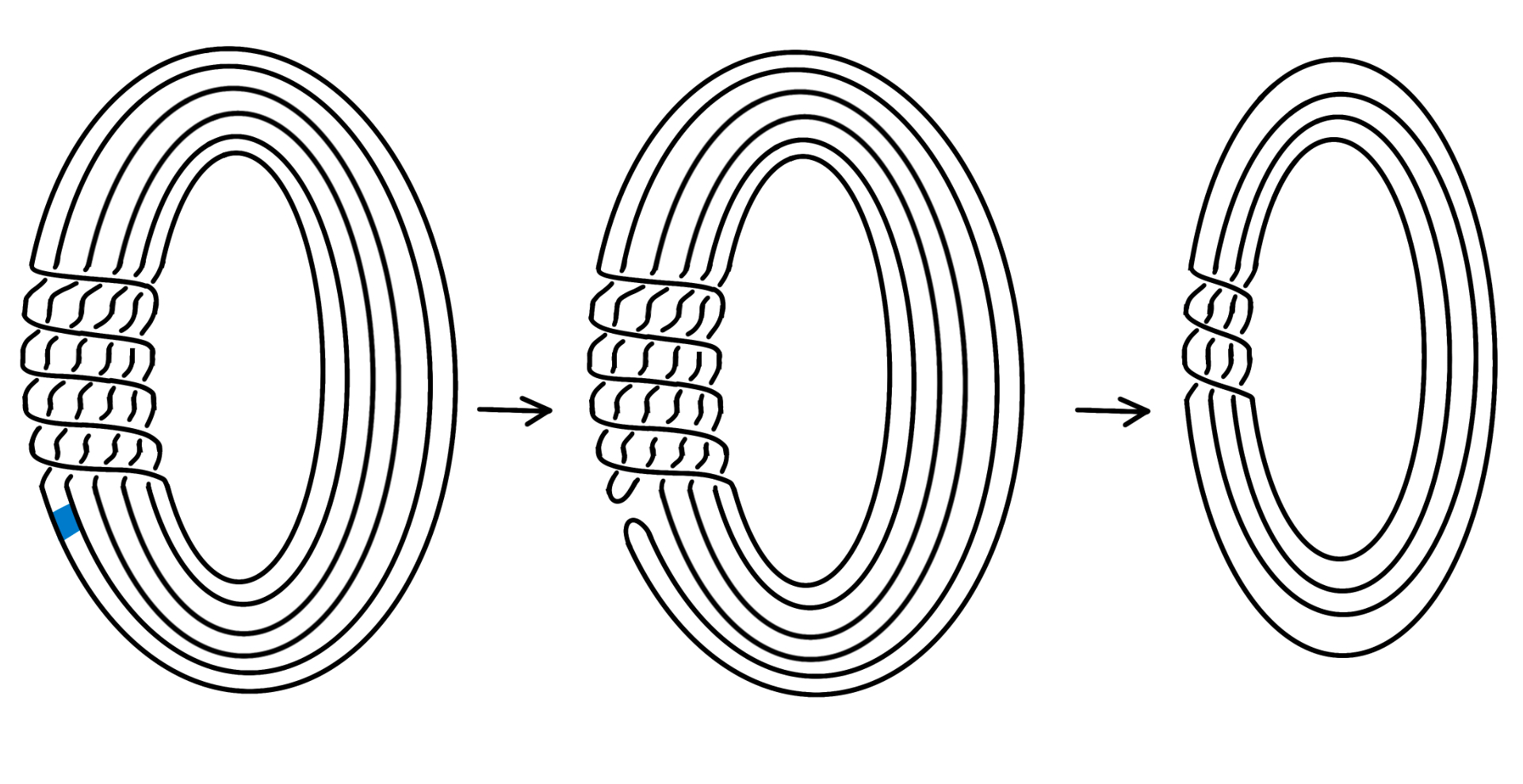}
    \caption{Non-Orientable Band Move from $T_{6, 5}$ to $T_{3, 4}$.}
    \label{fig:NOBM}
\end{figure}

\begin{proof}
    We apply a non-orientable band move to two adjacent strands of $T_{6,5}$ and $T_{6,17}$ to obtain torus knots $T_{4, q}$ that bound a M\"obius band, see Figure \ref{fig:NOBM}. Therefore $\gamma_4(T_{6,5})\leq 2$ and $\gamma_4(T_{6,17})\leq 2$ by Proposition~\ref{BMbound}. By Corollary~\ref{SigArf5mod12}, we know that $\gamma_4(T_{6,5})\geq 2$ and $\gamma_4(T_{6,17})\geq 2$.
\end{proof}

\begin{rem}
    The knot $T_{6, 13}$ can be transformed into $T_{4, 9}$ with one non-orientable band move and thus $\gamma_{4}(T_{6, 13}) \leq 2$.
\end{rem}

 We now compute the Upsilon invariant. Recall that we write $\upsilon(T_{p, q})$ as $\upsilon(p, q)$.
	\begin{prop}\label{upscalc}
		$$\upsilon(6,q)=\begin{cases}
			-9k   & \text{ if } q=6k+1\ge 0, \\
			-9k-6 & \text{ if } q=6k+5\ge 0,
		\end{cases}$$

        and

        $$\upsilon(5,q)=\begin{cases}
			-6k   & \text{ if } q=5k+1\ge 0, \\
        	-6k-2   & \text{ if } q=5k+2\ge 0, \\
        	-6k-3   & \text{ if } q=5k+3\ge 0, \\
			-6k-4 & \text{ if } q=5k+4\ge 0.
		\end{cases}$$
        
	\end{prop}

	\begin{proof}
		We will prove this for $\upsilon(6,q)$ by induction on $k$. For $k=0$ and $k=1$, we have 
\begin{align*}
    \upsilon(6,1)&=0,\\
    \upsilon(6,5)&=-6,\\
    \upsilon(6,7)&=-9.
\end{align*}

  For the inductive step, we apply Equation~\ref{upsilonrecursion} to get \[\upsilon(6,6k+ 1)=\upsilon(6,6(k-1)+ 1)+\upsilon(6,7)=-9(k-1)-9=-9k.\]
  and
  \[\upsilon(6,6k+ 5)=\upsilon(6,6(k-1)+ 5)+\upsilon(6,7)=-9(k-1)-6-9=-9k-6.\]
  The proof for $\upsilon(5,q)$ follows similarly.
  \end{proof}

Now we compute $\left| \upsilon(p,q) + \frac{\sigma(p,q)}{2} \right|$ for $p=5,6$ so that we can apply Theorem~\ref{UpsSigBound}. 

\begin{prop}\label{prop:upsSigCalc}
	\begin{equation}
		\left|\upsilon(6,q)+\dfrac{\sigma(6,q)}{2}\right| = \begin{cases}
			0 & \text{ if } q=6k+1, \\
			2 & \text{ if } q=6k+5,
		\end{cases}
	\end{equation}
    and
    \begin{equation}
        \left|\upsilon(5,q)+\dfrac{\sigma(5,q)}{2}\right| =
        \begin{cases}
			0 & \text{ if } q=10k+1, \\
            0 & \text{ if } q=10k+4, \\
            2 & \text{ if } q=10k+6, \\
            2 & \text{ if } q=10k+9, \\
            0 & \text{ if } q=5k+2, \\
			1 & \text{ if } q=5k+3.
		\end{cases}
	\end{equation}
\end{prop}

 \begin{proof}
		This follows from our computations of $\sigma(p,q)$ in Proposition~\ref{sigcalc} and of $\upsilon(p,q)$ in Proposition~\ref{upscalc}.
	\end{proof}

We get the following result by applying Theorem~\ref{UpsSigBound}.

\begin{cor}
    For any integer $k$, we have $\gamma_4(T_{6,6k+5})\geq2$, $\gamma_{4}(T_{5, 10k+6}) \geq 2$, and $\gamma_{4}(T_{5, 10k+9}) \geq 2$. 
\end{cor}

Note that this result improves on our previous lower bound $\gamma_4(T_{6, q})$ given by the the signature and Arf invariant in Corollary \ref{SigArf5mod12}. It does not give us new information on $\gamma_4(T_{5,q}$).

\subsection{Calculating the \textit{d}-invariant}

Recall that the Alexander polynomial for a torus knot is given by \[\Delta_{p, q} (t) = t^{-(p-1)(q-1)/2} \dfrac{(t-1)(t^{pq}-1)}{(t^{p}-1)(t^{q}-1)}.\]
Since the Alexander polynomial is symmetric, we know that the largest power of $\Delta_{p, q} (t)$ is $\frac{(p-1)(q-1)}{2}$.

For the computation of the $d$-invariant, we follow the conventions of Batson, though this was originally stated by Ozv\'ath and Szab\'o in \cite{OSref}.

\begin{prop}[Proposition 7 in \cite{batson}]\label{dinv}

	Since the Alexander polynomial is symmetric, we can express it as
    \begin{equation}\label{eqn:alex}
		\Delta_{p,q}(t)=a_0+\sum_{i=1}^d a_i(t^i+t^{-i}).
	\end{equation}
 Then the $d$ invariant for $T_{p,q}$ is given by

	\begin{equation}\label{eqn:dinv}
		d(p,q)=2t_0=2\sum_{i=1}^d ja_j. 
	\end{equation}
\end{prop}

Now we will use this to compute the $d$-invariant for general $T_{p,q}$ torus knots, which we denote $d(p, q)$. For the convenience of the reader, we restate Theorem \hyperref[dinvTHM]{1.1}.

{{\noindent\bfseries Theorem 1.1.}\label{dinvTHM} \textit{If} $p<q$,
\[ d(p,q)=2\left(\left\lfloor\dfrac{p}{2} \right\rfloor+\sum_{k=0}^{\lfloor\frac{p}{2} \rfloor-1}\left\lfloor\frac{(p-1-2k)q-p-1}{2p}\right\rfloor\right). \]}

\begin{proof}

    To prove this, we make use of the following two known equations.
    \begin{equation}\label{eqn}
        \dfrac{t^p-1}{t-1}=\sum_{i=0}^{p-1}t^i
    \end{equation}
    and
    \begin{equation}\label{eqn:inverse}
        \dfrac{t-1}{t^p-1}=\sum^{\infty}_{i=0}t^{ip}\delta 	
    \end{equation}
    where $\delta=1-t.$
    As we stated before, we will be considering $t^{(p-1)(q-1)/2}\Delta_{p, q}(t)$ for ease of notation. Using equations~\ref{eqn} and~\ref{eqn:inverse}, we get  
    \begin{equation}\label{doublesum}
     t^{(p-1)(q-1)/2}\Delta_{p, q}(t)=\dfrac{(t^{pq}-1)}{(t^{q}-1)}\dfrac{(t-1)}{(t^{p}-1)}=\sum^{p-1}_{k=0}t^{kq}\sum^{\infty}_{n=0}t^{np}\delta=\sum^{p-1}_{k=0}\sum^{\infty}_{n=0}t^{np+kq}\delta. 
     \end{equation}
 
    Note that when we multiply by $t^{(p-1)(q-1)/2}$, the constant term $c$ of $\Delta_{p,q}(t)$ becomes $c t^{(p-1)(q-1)/2}$ in $t^{(p-1)(q-1)/2}\Delta_{p, q}(t)$. Therefore, we only need to consider terms of Equation~\ref{doublesum} with powers less than $\frac{(p-1)(q-1)}{2}$. Multiplying $t^{np + kq}\delta$ we obtain $t^{np+kq} - t^{np+kq + 1}$, thus the highest power that we need to consider is $np + kq +1$. Hence, we have the inequality $np+kq\le \frac{(p-1)(q-1)}{2}-1$. To obtain a bound on $k$, we put $n=0$ to get $k\le \frac{p-1}{2}-\frac{p+1}{2q}$.

    Because $k$ is an integer, we apply the floor function to each side to get $k\le \lfloor \frac{p-1}{2}-\frac{p+1}{2q} \rfloor$. If $n$ is odd, $\frac{p-1}{2}$ is an integer; as $p<q$, the term $\frac{p+1}{2q}$ is less than $\frac{1}{2}$ but greater than $0$, so $k\le \lfloor\frac{p}{2} \rfloor-1$. Else, if $n$ is even, $\frac{p}{2}$ is an integer; the fractional part, by the previous, is less than $1$ but strictly greater than $\frac{1}{2}$, so we again get $k\le \lfloor\frac{p}{2} \rfloor-1$. Now that we have a bound on $k$, we need to know bounds for $n$. Solving for $n$ in terms of $k$ yields the inequality $n\le \lfloor\frac{(p-1-2k)q-p-1}{2p}\rfloor$.

    Applying this to the sum, we can then consider the simplified polynomial
    
    \[\sum^{\lfloor\frac{p}{2}\rfloor-1}_{k=0}\sum^{\lfloor\frac{(p-1-2k)q-p-1}{2p}\rfloor}_{n=0}t^{np+kq}\delta. \]

    From here, calculating the weighted sum of the coefficients via Equation \ref{eqn:dinv} amounts to counting the occurrences of $\delta$, as $t^m\delta=t^m-t^{m+1}$ contributes $-(m-(m+1))=1$ to the sum. Each internal sum amounts to $\lfloor\frac{(p-1-2k)q-p-1}{2p}\rfloor+1$ terms; pulling each $1$ out of the sum and multiplying by $2$ gives the $d$-invariant.

    \begin{align*}
        d(p,q)&=2\left(\sum^{\lfloor\frac{p}{2}\rfloor-1}_{k=0}\sum^{\lfloor\frac{(p-1-2k)q-p-1}{2p}\rfloor}_{n=0}1\right)
        \\&=2\left(
        \sum^{\lfloor\frac{p}{2}\rfloor-1}_{k=0}\left(
        \left\lfloor\frac{(p-1-2k)q-p-1}{2p}\right\rfloor+1
        \right)
        \right)\\
        &=2
        \left(\left\lfloor\frac{p}{2}\right\rfloor+\sum^{\lfloor\frac{p}{2}\rfloor-1}_{k=0}
        \left\lfloor\frac{
        (p-1-2k)q-p-1}{2p}\right\rfloor\right).
    \end{align*}
\end{proof}

For our specific uses, we use this formula (along with our base cases) to give us the following result. 
	
\begin{prop}
	\[ d(6,q)=\begin{cases}
		18k    &\text{ if }q=12k+1,\\
		18k+6  &\text{ if }q=12k+5,\\
		18k+12 &\text{ if }q=12k+7,\\
		18k+18 &\text{ if }q=12k+11,
	\end{cases}	\]
 and
	\[ d(5,q)=\begin{cases}
		6k    &\text{ if }q=5k+1,\\
		6k+2  &\text{ if }q=5k+2,\\
		6k+4 &\text{ if }q=5k+3,\\
		6k+6 &\text{ if }q=5k+4.
	\end{cases}	\]
\end{prop}

\begin{proof}
We will show this is true for $d(6,12k+1)$ as the rest of the cases follow similarly.
\begin{align*}
    d(6,12k+1)&=2
        \left(\left\lfloor\frac{6}{2}\right\rfloor+\sum^{\lfloor\frac{6}{2}\rfloor-1}_{j=0}
        \left\lfloor\frac{
        (6-1-2j)(12k+1)-6-1}{2(6)}\right\rfloor\right)\\
        &=2
        \left(3+\sum^{2}_{k=0}
        \left\lfloor\frac{
        (5-2j)(12k+1)-7}{12}\right\rfloor\right)\\
        &=2
        \left(3+
        \left\lfloor\frac{
        5(12k+1)-7}{12}\right\rfloor+\left\lfloor\frac{
        3(12k+1)-7}{12}\right\rfloor+\left\lfloor\frac{
        (12k+1)-7}{12}\right\rfloor\right)\\
        &=2
        \left(3+
        \left\lfloor\frac{
        5(12k)-2}{12}\right\rfloor+\left\lfloor\frac{
        3(12k)-4}{12}\right\rfloor+\left\lfloor\frac{
        (12k)-6}{12}\right\rfloor\right)\\
        &=2
        \left(3+9k+
        \left\lfloor\frac{
        -2}{12}\right\rfloor+\left\lfloor\frac{
        -4}{12}\right\rfloor+\left\lfloor\frac{
        -6}{12}\right\rfloor\right)\\
        &=2
        \left(3+9k-3\right)\\
        &=18k.
\end{align*}
\end{proof}

	We will now compute $\dfrac{\sigma(p,q)}{2}-d(p,q)$ for $p=5,6$ so that we can apply Theorem~\ref{boundsd}.
\begin{prop}
    	\begin{equation}\label{eqn:boundsd1}
		\dfrac{\sigma(6,q)}{2}-d(6,q)= \begin{cases}
			0&\text{ if }q=12k+1,\\
			2&\text{ if }q=12k+5,\\
			-3&\text{ if }q=12k+7,\\
			-1&\text{ if }q=12k+11,
		\end{cases}
    \end{equation}
and
    	\begin{equation}\label{eqn:boundsd2}
		\dfrac{\sigma(5,q)}{2}-d(5,q)= \begin{cases}
			0&\text{ if }q=10k+1,\\
			-2&\text{ if }q=10k+4,\\
			2&\text{ if }q=10k+6,\\
            0&\text{ if }q=10k+9,\\
            0&\text{ if }q=5k+2,\\
			0&\text{ if }q=5k+3.
		\end{cases}
	\end{equation}

\end{prop}

Note that this does not give us improved lower bounds on $\gamma_4(T_{p,q})$ for either $p=5$ or $p=6.$

\subsection{Linking Form Obstructions}

There is a useful obstruction to knots bounding M\"obius bands which comes from the linking form on the double branched cover of a knot over $S^{3}$, denoted $\DK$. We give the reader some brief background but we refer to Rolfsen's book \cite{rolfsen} to learn about double branched covers, a paper by Conway, Friedl, and Herrmann \cite{lkforms} to learn linking forms, and to Gilmer and Livingston \cite{GL} to learn the details of the obstruction. However, computing this obstruction mainly involves modular arithmetic, making this an approachable technique.

The linking form is a bilinear map on the first homology, $$\lambda : H_{1}(\DK ; \Z) \times H_{1} (\DK ; \Z) \to \Q/\Z. $$

\begin{cor}[Corollary 3 in \cite{GL}]\label{linkingNumThm}

Suppose that $H_{1}(D_{K}(S^{3})) = \mathbb{Z}_{n} $ where $n$ is the product of primes, all with odd exponent. Then if $K$ bounds a M\"{o}bius band in $B^{4}$, there is a generator $a \in H_{1} ( D_{K}(S^{3})) $ such that $ \lambda (a, a) = \pm \frac{1}{n}.$
    
\end{cor}

It is a fact that the double branched covers of $S^{3}$ branched over torus knots $T_{p, q}$ are Brieskorn spheres $\Sigma(2, p, q)$  when $p$ is even~\cite{torus1}.

\begin{lem}[Lemma 5.3 in \cite{torus1}]\label{homZq}
    $H_{1} (\Sigma(2, p, q); \Z) \cong \Z_{q}$.
\end{lem}

Observe that by combining Lemma \ref{homZq} and the definition of the linking form, we can rewrite $\lambda$ as a bilinear map $\lambda: \Z_{q} \times \Z_{q} \to \Q/ \Z$.

\begin{prop}[Proposition 5.2 in \cite{torus1}]\label{lkformTorusknots}
    For any relatively prime positive integers $p$ and $q$ with $p$ even, we have $\lambda (a, a) \equiv -\frac{p}{2q} \pmod{1} $ for some generator $a$ of $H_{1}(\Sigma(2, p, q); \Z)$.
\end{prop}

From Proposition \ref{lkformTorusknots}, we see that the linking form is $\lambda (a, a) = - \frac{p}{2q}$. We wish to determine whether or not there exists a integer $n$ so that $\lambda (na, na) = \pm \frac{1}{q}$ so that we can apply Corollary \ref{linkingNumThm} and Lemma \ref{homZq}. 

Supposing $\lambda (na, na) = \pm \frac{1}{q}$ we have $\lambda(na, na) = n^{2} \lambda(a, a) = n^{2} (- \frac{p}{2q}) $ by bilinearity. Thus, we only need to determine whether $- \frac{p n^{2}}{2}$ is congruent to $\pm 1 \pmod{q}$. If there is not a congruence, then we have satisfied the obstruction and $T_{p,q}$ does not bound a M\"obius band.

For $T_{6, q}$ torus knots, recall that we have an upper bound $\gamma_{4}(T_{6, q}) \leq 3$ for all $q$ and a lower bound $2 \leq \gamma_{4}(T_{6,q})$ for $q \equiv 5, 7, 11 \pmod{12}$. We wish to determine whether or not $T_{6,q}$  bounds a M\"obius band when $q \equiv 1 \pmod{12}$. Since $T_{6,1}$ is the unknot, we only consider $q > 1$.

\begin{lem}
If $q \equiv 0 \pmod{5}$, then $\gamma_{4}(T_{6, q}) \geq 2$.

\end{lem}
\begin{proof}
    Given $T_{6, q}$ where $q \equiv 1 \pmod{12}$, we consider $\lambda(a , a) = - \frac{6}{2q} = - \frac{3}{q}$. As outlined above, we check if $- 3n^{2}$ is congruent to $\pm 1 \pmod{q}$. We consider $q=5$ and note this encompasses the results for $q \equiv 0 \pmod{5}$. Suppose there exists some $n \in \Z$ so that $\lambda(na, na) = \pm\frac{1}{5}$. Then, $-\frac{3}{5}  = n^{2} (\pm \frac{1}{5})$, which is not possible, as this would imply $-3n^{2} \equiv \pm 1 \pmod{5}$. Thus, we have met the obstruction in Corollary \ref{linkingNumThm}.
\end{proof}

\begin{rem}
    For $K = T_{6, q}$ where $q \equiv 1 \pmod{12}$, the linking form on $\DK$ meets the linking form obstruction when $q \equiv 0 \pmod{5}$, but does not meet the linking form obstruction for all $q\not\equiv 0\pmod{5}$. As these torus knots in this latter case have not met any other obstruction mentioned in the paper, the best we can say is that $\gamma_{4}(T_{6, q}) \geq 1$, meaning we do not yet know whether they bound a M\"obius band. 
\end{rem}

\newpage
\nocite{*}
\bibliography{torus_refs}
\bibliographystyle{plain}

\end{document}